\documentclass[12pt]{article}

\usepackage[T1]{fontenc} \usepackage[english]{babel}
\usepackage{latexsym, amsmath, amssymb, amsthm}

\usepackage{pgf,tikz}
\usetikzlibrary{arrows}

\begin{document}

\newtheorem{theorem}{Theorem}
\newtheorem{lemma}[theorem]{Lemma}
\newtheorem{corollary}[theorem]{Corollary}
\newtheorem{question}[theorem]{Question}
\newtheorem{proposition}[theorem]{Proposition}
\theoremstyle{definition}
\newtheorem{definition}[theorem]{Definition}
\newtheorem{example}[theorem]{Example}

\newcommand{\R}{\mbox{${\bf R}$}}
\newcommand{\Z}{\mbox{${\bf Z}$}}
\newcommand{\Q}{\mbox{${\bf Q}$}}
\newcommand{\N}{\mbox{${\bf N}$}}
\newcommand{\elt}{\mbox{$v \in G$}}
\newcommand{\elg}{\mbox{$\phi \in \Gamma$}}
\newcommand{\gstab}{\mbox{$G_v$}}
\newcommand{\caa}{\mbox{$C_1(g)$}}
\newcommand{\cia}{\mbox{$C_i(g)$}}
\newcommand{\ima}{\mbox{${\rm Im}$}}
\newcommand{\de}{\mbox{${\rm deg}$}}
\newcommand{\dist}{\mbox{${\rm dist}$}}
\newcommand{\DL}{\mbox{${\rm\sf DL}$}}
\newcommand{\wre}{\mbox{${\rm Wr\,}$}}
\newcommand{\res}{\mbox{$\mid$}}
\newcommand{\nin}{\mbox{$\ \not\in\ $}}
\newcommand{\fix}{\mbox{${\rm fix}$}}
\newcommand{\pr}{\mbox{${\rm pr}$}}
\newcommand{\sym}{\mbox{${\rm Sym}\;$}}
\newcommand{\aut}{\mbox{${\rm Aut}$}}
\newcommand{\auttn}[0]{\mbox{${\rm Aut}\;T_n$}}
\newcommand{\autx}[0]{\mbox{${\rm Aut}\;X$}}
\newcommand{\autga}[0]{\mbox{${\rm Aut}(\Gamma)$}}
\newcommand{\autde}[0]{\mbox{${\rm Aut}(\Delta)$}}
\newcommand{\out}{\mbox{${\rm out}$}}
\newcommand{\inn}{\mbox{${\rm in}$}}
\newcommand{\Al}{\mbox{${\rm Al}$}}
\newcommand{\Pl}{\mbox{${\rm Pl}$}}

\newcommand{\vZZ}{\vec{\mathbb{Z}}}

\renewcommand{\div}{{\hbox { \rm div }}}
\renewcommand{\mod}{{\hbox { \rm mod }}}
\newcommand{\id}{\hbox{\rm id}}
\newcommand{\ZZ}{\mathbb{Z}}
\newcommand{\Aut}{\mathrm{Aut}}
\newcommand{\tG}{{\widetilde{G}}}
\newcommand{\tH}{{\widetilde{H}}}
\newcommand{\tHpsi}{{\widetilde{H}}_\psi}
\newcommand{\CDC}{\mathop{{\rm CDC}}}
\newcommand{\CDHC}{\mathop{{\rm CDHC}}}
\newcommand{\g}{\mathbf g}
\newcommand{\h}{\mathbf h}
\newcommand{\V}{\mathrm V}
\newcommand{\VX}{{\mathrm V}X}
\newcommand{\E}{\mathrm E}
\newcommand{\A}{\mathrm A}
\newcommand{\TFaut}{{\mathop{{\rm TF}}}}
\newcommand{\Sym}{\mathrm{Sym}}
\newcommand{\Haut}{{\mathop{{\rm HAut}}}}
\newcommand{\Hautde}[0]{\mbox{${\rm HAut}(\Delta)$}}

\title{Sharply $k$-arc-transitive-digraphs:\\
finite and infinite examples}

\author{
R\"ognvaldur G.\ M\"oller\thanks{Science Institute, University of Iceland, IS-107 Reykjav\'ik, Iceland.  E-mail: {\tt roggi@hi.is} \newline R\"ognvaldur G.\ M\"oller acknowledges support from the University of Iceland Research Fund.}
 ,\quad
Primo\v{z} Poto\v{c}nik\thanks{Faculty of Mathematics and Physics, University of Ljubljana, Jadranska 19, SI-1000 Ljubljana, Slovenia. E-mail: {\tt primoz.potocnik@fmf.uni-lj.si} \newline
Primo\v{z} Poto\v{c}nik acknowledges the financial support from the Slovenian Research Agency (research core funding No. P1-0294).} ,\quad
Norbert Seifter\thanks{Montanuniversit\"at~Leoben, Franz-Josef-Strasse~18, A-8700 Leoben,~Austria.
E-mail: {\tt seifter@unileoben.ac.at}}
}

\maketitle

\begin{abstract}
\noindent
A general method for constructing sharply $k$-arc-transitive digraphs, i.e.~digraphs that are $k$-arc-transitive but not $(k+1)$-arc-transitive, is presented.  Using our method it is possible to construct both finite and infinite examples.  The infinite examples can have one, two or infinitely many ends.  Among the one-ended examples there are also digraphs that have polynomial growth.   
\end{abstract}

\section*{Introduction}

For an integer $k\geq 1$ a {\em $k$-arc} of a digraph $\Gamma$
is a sequence $(v_0,\ldots ,v_k)$ of $k+1$ vertices of $\Gamma$
such that $(v_{i-1},v_i)$ is an arc
of $\Gamma$ for each $i$, $1\leq i\leq k$. We say that $\Gamma$ is $k$-arc-transitive if
$\Aut(\Gamma)$ acts transitively on the set of $k$-arcs of $\Gamma$ and {\em sharply $k$-arc-transitive} if $G$ is $k$-arc-transitive but not $(k+1)$-arc-transitive.  In an undirected graph a $k$-arc is a sequence of $k+1$ vertices such that $v_{i-1}$ and $v_i$ are adjacent for $1\leq i\leq k$ and $v_{i-2}\neq v_i$ if $2\leq i\leq k$.  The concept of  $k$-arc transitivity
was first introduced and investigated for (undirected) finite graphs, culminating in the striking result of R. Weiss \cite{Weiss1981}
which states that a finite graph with valency $\geq 3$ cannot be $8$-arc-transitive.  This result also holds for infinite graphs with polynomial growth as was shown in \cite{Seifter1991}.  On the other hand,  Hamann and Pott \cite[Corollary 1.2]{HamannPott2012} have shown that a 2-arc-transitive connected undirected graph with more than one end must be a regular tree (this was first proved for locally finite graphs by Thomassen and Woess \cite[Theorem~3.2]{ThomassenWoess1993}).

For digraphs the situation is more complicated.  There exist several constructions of finite $k$-arc-transitive digraphs
for arbitrarily large $k$ (see e.g. \cite{ConderLorimerPraeger1995,Praeger1989,MansillaSerra2001}). If the digraphs in consideration 
are infinite, then they might even be highly-arc-transitive, which means that the automorphism group is $k$-arc-transitive for all $k\geq 0$, where $0$-arc-transitivity means  that the digraph is vertex-transitive. The concept of highly-arc-transitive digraphs was introduced in \cite{CPW1993}.   
For infinite digraphs with more than one end it seems 
at a first glance that $k$-arc-transitivity, for sufficiently large $k$, might always imply that the digraph is also highly-arc-transitive. As was shown in \cite{Seifter2008} this really holds for certain digraphs with more than one end and prime in- and out-valency and further results along these lines are proved in \cite{MollerPotocnikSeifter2018}.  But in general this does not hold for  digraphs with more than one end, as was shown by Mansilla \cite{Mansilla2007}
who constructed two-ended 2-arc-transitive digraphs that are not 3-arc-transitive.  

In this paper we present a general method for constructing sharply $k$-arc-transitive digraphs.  The basic construction takes as an input a digraph $\Delta$ that satisfies mild 
symmetry conditions and an integer $k$ and produces a sharply $k$-arc-transitive digraph.  When $\Delta$ is finite the construction produces a two-ended digraph and by varying $\Delta$ one gets such digraphs with any possible valency.  By choosing $\Delta$ to be an infinite	 directed line with a loop attached at every vertex we get a sharply $k$-arc-transitive one-ended digraph with polynomial growth.  Variations of the construction can be used to produce finite sharply $k$-arc-transitive digraphs as well as infinite sharply $k$-arc-transitive digraphs with exponential growth and one or infinitely many ends.

\medskip

\noindent
{\bf Remark.}  When we were nearing the completion of the project described in this paper we realized that our work is very closely related to the work done by Cheryl Praeger in \cite{Praeger1989}.  Praeger only considers finite digraphs but one can show that the basic idea in her construction \cite[Definition~2.10]{Praeger1989} can be modified to yield infinite digraphs that are isomorphic to the digraphs we construct in Section~\ref{SConstruction}.  Praeger's set up is not as general as ours and while she proves that her digraphs $C_r(v, s,\Delta)$ are $s$-arc-transitive she does not consider the question if these graphs are sharply $s$-arc-transitive.  In her paper she does construct a family $C_r(v,s)$ of finite digraphs that are sharply $s$-arc-transitive but the reason why these graphs are not $(s+1)$-arc-transitive is entirely different from the reason why our graphs are not $(s+1)$-arc-transitive. 
 The connections between our work and Praeger's work in \cite{Praeger1989} are described in more details in remarks in Sections~\ref{SConstruction}, \ref{SAutomorphism} and \ref{SFinite}.

\section*{Notation and terminology}

A {\em digraph} $\Gamma$ is a pair of sets $(\V\Gamma, \E\Gamma)$ where $\V\Gamma$ is a set whose elements we call {\em vertices} and $\E\Gamma$ is a subset of $\V\Gamma\times \V\Gamma$ whose elements we call {\em arcs}.
If $(u,v)$ is an arc of $\Gamma$, then $u$ is called the {\em initial vertex} and $v$ the {\em terminal vertex} of the arc $(u,v)$.
 In some examples we will use digraphs with loops, i.e.~arcs of the type $(v,v)$,  and we also allow, if $v\neq w$, that both $(v,w)$ and $(w,v)$ are arcs.  An {\em undirected graph}, or just {\em graph}, is similarly a pair $(\V\Gamma, \E\Gamma)$ but now $\E\Gamma$ is a set of 1 and 2 element subsets of $\V\Gamma$.  The elements of $\E\Gamma$ are called {\em edges}. For a digraph $\Gamma$ we define the {\em underlying undirected graph} as the undirected graph that has the same vertex set as $\Gamma$ and the edge set consisting of all sets $\{u, v\}$ where $u$ and $v$ are vertices in $\Gamma$ and $(u,v)$ or $(v,u)$ is an arc in $\Gamma$.

For a vertex $v$ in a digraph $\Gamma$
we define the sets of {\em in-} and {\em out-neighbours} as
$$\inn(v)=\{u\in \V\Gamma\mid (u,v)\in \E\Gamma\}\qquad\mbox{and}\qquad
\out(v)=\{u\in \V\Gamma\mid (v,u)\in \E\Gamma\}.$$
A vertex that is an in- or out-neighbour of $v$ is said to be a {\em neighbour}.
The cardinality of $\inn(v)$ is the {\em in-valency} of $v$ and the cardinality of $\out(v)$ is the {\em out-valency} of $v$.

A {\em path} $W$ in a digraph $\Gamma$ is a sequence $v_0, a_1, v_1, a_2, \ldots, a_n, v_n$ where the $v_i$'s are pairwise distinct vertices and the $a_i$'s are arcs such that $a_i=(v_{i-1}, v_{i})$ or $a_i=(v_i, v_{i-1})$  for $i=0, 1, \ldots, n$.  One-way infinite paths $v_0, a_1, v_1, a_2, v_3, \ldots$ are often called {\em rays}.   When the digraph is without loops and asymmetric, i.e.~there is no pair of vertices $u, v$ such that both $(u,v)$ and $(v,u)$ are arcs, then one can leave out the arcs when discussing  paths and just list the vertices.   A digraph is said to be {\em connected} if for any pair $v, w$ of vertices there is a path with initial vertex $v$ and terminal vertex $w$.  

A digraph is said to have two ends if it is connected and by removing some finite set of vertices it is possible to get two infinite components, but one never gets more than two infinite components.   Formally the ends of a digraph (or a graph) are defined as equivalence classes of rays where two rays are said to be equivalent, or belong to the same end, if there exists an infinite family of pairwise disjoint paths all with their initial vertex in one of the rays and the terminal vertex in the other.  This equivalence relation is clearly invariant under the action of the automorphism group and thus the automorphism group acts on the equivalence classes, i.e.~the ends.   Note that by \cite[Corollary~4]{DiestelJungMoller1993} a vertex transitive graph has either 0, 1, 2 or at least $2^{\aleph_0}$ ends.
  
An {\em automorphism} of a digraph $\Gamma$ is a bijective map $\varphi: \V\Gamma\to \V\Gamma$ such that $(u,v)$ is an arc in $\Gamma$ if and only if $(\varphi(u), \varphi(v))$ is an arc in $\Gamma$.  The group of all automorphisms of $\Gamma$ is denoted by $\autga$.  A {\em digraph morphism} from a digraph $\Gamma_1$ to a digraph $\Gamma_2$ is a map $\varphi: \V\Gamma_1\to \V\Gamma_2$ such that if $(u,v)$ is an arc in $\Gamma_1$ then $(\varphi(u), \varphi(v))$ is an arc in $\Gamma_2$.
A {\em $k$-arc} in a digraph
$\Gamma$ is a sequence of vertices $v_0, v_1, \ldots, v_k$ such that $(v_i, v_{i+1})$ is an arc for $i=0, 1, \ldots, k-1$.
The vertex $v_0$ is the {\em initial vertex} of the $k$-arc and $v_k$ the {\em terminal vertex}.
The group of automorphisms acts on the set of $k$-arcs and if this action is transitive then we say that $\Gamma$ is {\em $k$-arc-transitive}.  (Note that 0-arc-transitive means just that the digraph is vertex-transitive.)   A digraph that is $k$-arc-transitive but not $(k+1)$-arc-transitive is said to be {\em sharply $k$-arc-transitive}.  If a digraph is $k$-arc-transitive for all $k$ then we say that it is {\em highly-arc-transitive}.

Two arcs (not necessarily distinct) are said to be {\em related} if they have a common initial vertex  or a common terminal vertex.  Let $R$ be the transitive closure of this relation.  The relation $R$ is clearly an
$\autga$-invariant equivalence relation.  
 A  subdigraph spanned by one of the equivalence classes is called an {\em alternet}.  If the automorphism group is transitive on arcs then all the alternets are isomorphic.  This relation is defined in the Introduction of \cite{CPW1993} where it is called the {\em reachability relation} and the alternets are called {\em reachability digraphs}.

One can regard the integers $\ZZ$ as an undirected graph in an obvious way, but it is also possible to regard them as a digraph $\vZZ$ with
arc-set $\{(i,i+1) \mid i \in \ZZ\}$.  A digraph $\Gamma$ is said to have {\em Property Z} if there exists a surjective digraph morphism $\Gamma\to \vZZ$.

Consider now a group $G$ acting on a set $\Omega$.
Denote the image of a point $x\in \Omega$ under an element $g\in G$ by $x^g$.  The {\em stabilizer of a point $a$ in} $G$ is the subgroup $G_a=\{g\in G\mid a^g=a\}$.  For a set $A\subseteq \Omega$ we denote the {\em pointwise stabilizer} of $A$ by
$G_{(A)}=\{g\in G\mid a^g=a\mbox{ for all }a\in A\}$ and the {\em setwise stabilizer} by $G_{\{A\}}=\{g\in G\mid A^g=A\}$.   

For a group $G$ with generating set $S$ we define the {\em Cayley digraph} of $G$ with respect to the generating set $S$ as the digraph with vertex set $G$ and $(g,h)$ is an arc if $h=sg$ for some element $s\in S$.  If a group  $G$ acts regularly (i.e.~transitively and no group element, except the identity, fixes any vertex) on a connected digraph $\Delta$ then $\Delta$ is a Cayley digraph for $G$ with respect to some generating set $S$.

\section{General construction}\label{SConstruction}

Let $\Delta$ be a digraph (possibly with loops) with vertex set $V=\V\Delta$ and arc set $\E\Delta$.
 For a positive integer $k$ define
$\ZZ(\Delta,k)$ as the digraph with vertex set $\ZZ \times V^k$ in which the out-neighbourhood of a vertex
$(i;v_0, \ldots, v_{k-1})$ is the set of vertices $(i+1;w_0, \ldots, w_{k-1})$ satisfying:
$$
 v_j = w_j \> \hbox { if } \> j \not \equiv i \mod k, \qquad\mbox{and}\qquad
 (v_j, w_j)\in \E\Delta \> \hbox { if } \>  j \equiv i \mod k.
$$
Set $\Gamma=\ZZ(\Delta, k)$.  The map $\varphi:\V\Gamma\to\vZZ$ such that $\varphi(i;v_0, \ldots, v_{k-1})=i$ is a digraph morphism and thus $\Gamma$ has Property Z. For $i\in \ZZ$, let $\Gamma_i = \{i\} \times V^k=\varphi^{-1}(i)$.

The {\em canonical double cover} of a digraph $\Delta$, denoted with $\CDC(\Delta)$, is the digraph with vertex-set $V\times \ZZ_2$ (for a non-negative integer $n$ we let $\ZZ_n$ denote the ring of integers modulo $n$) such that $(x,y)$ is an arc in $\Delta$ if and only if
%
%
$((x,0),(y,1))$  and $((x,1), (y,0))$ are arcs in $\CDC(\Delta)$.  In the present work we have use for the subdigraph of $\CDC(\Delta)$ that has the same vertex set as $\CDC(\Delta)$ and the arc set is the set of all arcs in $\CDC(\Delta)$ of the type $((x,0),(y,1))$.  This subdigraph will be called the {\em canonical double half-cover} of $\Delta$ and denoted $\CDHC(\Delta)$.  If $\Delta$ has the property that whenever $(u,v)$ is an arc in $\Delta$ then $(v,u)$ is also an arc in $\Delta$ then the underlying undirected graph of $\CDHC(\Delta)$ is the canonical double cover of the underlying undirected graph of $\Delta$.
It is easy to see that in this case $\CDC(\Delta)$ (and thus also $\CDHC(\Delta)$) is connected if and only if $\Delta$ is connected and 
not bipartite.  

Let $i$ and $i'$ be integers such that $0\leq i'\leq k-1$ and $i'\equiv i \mod k$.  For a given choice of vertices $v_0, \ldots, v_{i'-1}, v_{i'+1}, \ldots, v_k$ of $\Delta$ the subdigraph of $\Gamma$ spanned by the set of vertices
$$\{(i + \epsilon ;v_0, \ldots, v_{i'-1}, u, v_{i'+1}, \ldots, v_{k-1})  \mid u\in V, \epsilon \in \{0,1\}\}$$
is isomorphic
to the canonical double half-cover $\CDHC(\Delta)$.
When $\CDHC(\Delta)$ is connected these are just the
alternets of $\ZZ(\Delta,k)$ and then one sees that $\ZZ(\Delta, k)$ is connected.

\medskip

\noindent
{\bf Remark.}
As mentioned in the introduction,
Praeger constructs in \cite[Definition 2.10]{Praeger1989} a family $C_r(v,s,\Delta)$ of finite graphs, where $r, v, s$ are positive integers and $r$ is a multiple of $s$ and $\Delta$ is an undirected graph
with vertex-set $\ZZ_v$.  
The graph $\ZZ(\Delta,k)$ we defined above can be viewed as a generalization of her construction to the infinite case.
Praeger then shows in \cite[Proposition 2.11]{Praeger1989} that if $\Delta$ is 1-arc-transitive then the graph $C_r(v,s,\Delta)$ is $s$-arc-transitive.  In the case were $\Delta=K_v$ is a complete graph she also proves that the graph $C_r(v, r-s, \Delta)$ is sharply $s$-arc transitive.

In this paper, we not only generalize this to the infinite case by proving analogous results for $\ZZ(\Delta,k)$  but also relax the condition on $1$-arc-transitivity of $\Delta$ (see Lemma~\ref{LArc-transitive}), which enables us to construct a wide variety of sharply $k$-arc-transitive digraphs. Some of these examples have one end, others have two or infinitely many ends and some have polynomial growth of degree $k+1$ (see Corollary~\ref{CPolynomial}). Below we explain the relationship between the graphs $C_r(v,s,\Delta)$
and $\ZZ(\Delta,k)$ in more details.

Let $\Delta$ be an undirected graph 
with vertex set $\ZZ_v$ for some integer $v\ge 2$.  Let $k\ge 3$ be an integer and  $m$ an multiple of $k$.
Praeger defines the digraph $C_m(v,k,\Delta)$ with   vertex set 
$\ZZ_m\times (\V\Delta)^k = \ZZ_m\times (\ZZ_v)^k$ and with 
$((i;x_1, \ldots, x_{k-1}, x_k), (j;y_1, \ldots, y_{k-1}, y_k))$ being an arc of $C_m(v,k,\Delta)$ if and only if
$j=i+1$ (in $\ZZ_m$), $y_1$ is a neighbour of $x_k$ in $\Delta$ and $x_l=y_{l+1}$ for $l=1, \ldots, k-1$.
This construction can be generalized to yield a digraph
$P(\Delta, k)$ with  vertex set $\ZZ\times (\V\Delta)^k$ and $((i;x_1, \ldots, x_{k-1}, x_k), (j;y_1, \ldots, y_{k-1}, y_k))$ being an arc of $P(\Delta, k)$ if and only if $j=i+1$ (in $\ZZ$), $y_1$ is a neighbour of $x_k$ in $\Delta$ and $x_l=y_{l+1}$ for $l=1, \ldots, k-1$. We now show that $P(\Delta,k)$ is isomorphic to $\ZZ(\Delta,k)$.  (When we construct $\ZZ(\Delta,k)$ we think of each edge $\{v, w\}$ in $\Delta$ as being represented by two arcs $(v,w)$ and $(w,v)$.)

For an integer $i$ we let $i'$ be the integer such that 
$0\leq i'\leq k-1$ and $i'\equiv i \mod k$.
We define a map $\theta$ from $\V\ZZ(\Delta, k)$ to $\V P(\Delta, k)=\V\ZZ(\Delta, k)$ in such a way that if $i'=0$ then 
$$\theta(i;x_0, \ldots, x_{k-2}, x_{k-1})=  
(i;x_{k-1}, x_{k-2},\ldots,  x_{0})$$
and if $i'\geq 1$ then
$$
\theta(i;x_0, \ldots, x_{i'-1}, x_{i'}, x_{i'+1}, \ldots, x_{k-1})=(i;x_{i'-1}, \ldots,x_0, x_{k-1},\ldots, x_{i'+1}, x_{i'}).
$$  
The arcs in $\ZZ(\Delta, k)$ are of the type
 $$((i;x_1, \ldots, x_{i'-1}, x_{i'}, x_{i'+1}, \ldots, x_k), (i+1;x_1, \ldots, x_{i'-1}, y_{i'}, x_{i'+1},  \ldots, x_k))$$
  with $(x_{i'},y_{i'})$ being an arc of $\Delta$.  The image of this arc under $\theta$ is
$$((i;x_{i'-1}, \ldots, x_0, x_{k-1},\ldots, x_{i'}), (i+1;y_{i'}, \ldots, x_0, x_{k-1},\ldots, x_{i'+1}))$$
which is an arc in $P(\Delta, k)$. 
  The map $\theta$ is clearly bijective and $\theta^2$ is the identity map.  Hence $\theta$ is an isomorphism of the digraphs $\ZZ(\Delta, k)$ and $P(\Delta, k)$.  


\medskip

\section{The automorphism group}\label{SAutomorphism}

We now discuss the automorphisms of $\ZZ(\Delta,k)$. First, it is a matter of straightforward calculation to
check that  the permutation $s$ of
$\ZZ \times V^k$ defined by
$$
(i;x_0, \ldots, x_{k-2}, x_{k-1})^s = (i+1;x_{k-1}, x_0, \ldots, x_{k-2})
$$
is an automorphism of $\ZZ(\Delta,k)$.

For an automorphism $g\in\autde$ and $j\in \ZZ_k$,
let $[g]_j$ be the permutation of $\ZZ \times V^k$ defined by
$$(i;x_0,  \ldots, x_{j-1}, x_j, x_{j+1}, \ldots, x_{k-1})^{[g]_j} 
= (i;x_0, \ldots, x_{j-1}, x_j^{g}, x_{j+1}, \ldots, x_{k-1}).$$
Observe that $[g]_j$ is an automorphism of $\ZZ(\Delta,k)$.   A simple calculation yields that $s^{-1}[g]_js=[g]_{j+1}$. 
The group generated by $s$ and the set $\{[g]_j\mid g\in \autde,\ j\in \ZZ_k\}$
is isomorphic to the semidirect product $(\autde)^k \rtimes \ZZ$ where the automorphism $s$ generates $\ZZ$
and cyclically permutes the components of $(\autde)^k$.   

\begin{proposition}\label{PCayley}
Suppose $\Delta$ is a Cayley digraph for a group $H$.  If the digraph $\ZZ(\Delta,k)$  is connected then it is a Cayley digraph for the group $H^k \rtimes \ZZ$.  
\end{proposition}

\begin{proof}
Clearly $H^k \rtimes \ZZ$ acts regularly on $\ZZ(\Delta,k)$ and thus $\ZZ(\Delta,k)$ is a Cayley digraph for $G$.
\end{proof}

In some cases there is a more general way to construct automorphisms. For that we need the notion of a {\rm two-fold automorphism}
described below.


Let $\Delta$ be a digraph and $V$ the vertex-set of $\Delta$.
Following \cite{LauriMizziScapellato2011} and \cite{LauriMizziScapellato2015}, a pair  $(g,h)$ of permutations
of $V$ is called a {\it two-fold automorphism} (or a TF-automorphism, for short)
of $\Delta$ provided that for every
two vertices $u,v\in V$ we have that $(u,v)\in \E\Delta$ if and only if $(u^g, v^h)\in \E\Delta$.
TF-automorphisms were first studied in \cite{Zelinka1971}, where they were called {\it autotopies}.

TF-automorphisms arise naturally in the study of the automorphism group of the canonical double half-cover
of a digraph. Namely, if $\Delta$ is a digraph and $\CDHC(\Delta)$ is its canonical double half-cover, then
a TF-automorphism $(g,h)$ of $\Delta$ gives rise to an automorphism of $\CDHC(\Delta)$ that maps
the vertex $(x,0)$ to $(x^g,0)$ and the vertex $(x,1)$ to $(x^h,1)$ for every vertex $x$ of $\Delta$. Conversely,
every automorphism of $\CDHC(\Delta)$ arises in this way from a TF-automorphism of $\Delta$.

The set of all TF-automorphisms of $\Delta$ will be denoted by $\TFaut(\Delta)$ and it
forms a subgroup of $\Sym(V)\times\Sym(V)$. Let
$$
\TFaut_1(\Delta) =  \{ g \in \Sym(V) \mid \exists h\in \Sym(V) \hbox { such that } (g,h)\in \TFaut(\Delta)\}
$$
and
$$
\TFaut_2(\Delta)  =  \{ h \in \Sym(V) \mid \exists g\in \Sym(V) \hbox { such that } (g,h)\in \TFaut(\Delta)\}.
$$
Note that both $\TFaut_1(\Delta)$ and $\TFaut_2(\Delta)$ are subgroups of $\Sym(V)$. Let $\Haut(\Delta)$
denote their intersection.  Clearly, $\autde\leq \Haut(\Delta)$.

Below, \lq\lq$\!\div\!$\rq\rq\ denotes integer division so that if $n$ and $k$ are positive integers then \lq\lq$n\div k$\lq\lq\ denotes the largest integer $l$ such that $kl\leq n$.

\begin{lemma}\label{LAction}
Let $\g=\{g_t\}_{t\in\ZZ}$ be a family of permutations of the vertex set of a digraph $\Delta$ such that for each $t\in \ZZ$ the pair $(g_t, g_{t+1})$ is a two-fold automorphism of $\Delta$.   Then the mapping $[\g]_{j}$, for $0\leq j\leq k-1$, defined by
$$(i;x_0, \ldots, x_j, \ldots, x_{k-1})^{[\g]_{j}} = (i;x_0, \ldots,  x_j^{g_t},  \ldots, x_{k-1}),$$ 
where $t = (i-j+k-1) \div k$,
is an automorphism of  $\ZZ(\Delta,k)$.  
\end{lemma}

\begin{proof}
Suppose that
$$
 a=\bigl((i;x_0, \ldots, x_{k-1}),\, (i+1;y_0, \ldots, y_{k-1})\bigr)
 $$
  is an arc of $\ZZ(\Delta,k)$.
Set $t=(i-j+k-1) \div k$. Let $i'$ be an integer such that $i'\equiv i \mod k$ and $0\leq i'\leq k-1$.
Then $x_l=y_l$ if $l\neq i'$ and $(x_{i'}, y_{i'})$ is an arc in $\Delta$.

Suppose first that $j \not\equiv i \mod k$. Then 
$$(i-j+k-1) \div k = ((i+1)-j+k-1) \div k,$$ 
and hence
$$
a^{[\g]_{j}}= \bigl((i;x_0, \ldots,  x_j^{g_t},  \ldots, x_{k-1}),\,
(i +1;y_0, \ldots,  y_j^{g_t},  \ldots, y_{k-1})\bigr).
$$
Since $x_j=y_j$ it is clear that this is also an arc in $\ZZ(\Delta,k)$.

Suppose now that $j\equiv i \mod k$. Then $ ((i+1)-j+k-1) \div k = t + 1$, and hence
$$
a^{[\g]_{j}}=  \bigl((i;x_0, \ldots,  x_j^{g_t},  \ldots, x_{k-1}),\,
(i +1;y_0, \ldots,  y_j^{g_{{t+1}}},  \ldots, y_{k-1})\bigr).
$$
Since $a$ is an arc of $\ZZ(\Delta,k)$, we have  that $(x_j, y_j)=(x_{i'}, y_{i'})\in \E\Delta$.
And, because $(g_t, g_{t+1})$ is a TF-automorphism then $(x_j^{g_t}, y_j^{g_{t+1}})$ is also in $\E\Delta$.  Thus, the image of $a$ under $[\g]_{j}$
 is also an arc of $\ZZ(\Delta,k)$.
 \end{proof}

\begin{definition}
\label{def:TF12}
Suppose that  $H$ is a non-trivial subgroup of $\Haut(\Delta)$.  If $\psi$ is an automorphism of $H$
such that $(g,g^\psi) \in \TFaut(\Delta)$ for every $g\in H$ then
we say that $H$ is a {\it $\psi$-stable subgroup} of $\Haut(\Delta)$.

A $\psi$-stable subgroup $H$ is said to be {\it $\psi$-arc-transitive} if
for every two arcs
$(x,y)$ and $(z,w)$ of $\Delta$ there exists $g\in H$ such that $(x^g,y^{g^\psi}) = (z,w)$.
\end{definition}

Note that if $H$ is a $\psi$-stable subgroup of $\Hautde$ and $g\in H$, then $g^{\psi^t} \in H$
for every $t\in \ZZ$ and $(g^{\psi^t},g^{\psi^{t+1}}) \in \TFaut(\Delta)$ for every $t\in \ZZ$.

We now give two basic examples of digraphs admitting a $\psi$-arc-transitive $\psi$-stable subgroup.

\begin{example}\label{Eid-transitive}
Every subgroup $H\le \autde$ is an $\id_H$-stable subgroup of $\Hautde$.  If, in addition, $H$ acts
transitively on the arcs of $\Delta$, then it is $\id_H$-arc-transitive in the sense of Definition~\ref{def:TF12}.
\end{example}

\begin{example}\label{Ecycles}
\label{ex:cic}
Let $\Theta_n$ be a directed cycle of length $n\geq 3$ with a loop attached at every vertex.  Furthermore, $\Theta_\infty$ denotes an infinite directed line with a loop attached at every vertex.  Identify the vertex set of $\Theta_n$ with $\ZZ_n$ and the vertex set of $\Theta_\infty$ with $\ZZ$ in the obvious way.
Let $a$ and $b$ be the permutations of $\V\Theta_n$   with $3\leq n\leq \infty$ defined by $x^a = x+1$ and $x^b = - x$, respectively.
Set $H=\langle a, b\rangle$ and observe that $H$ is the dihedral group consisting of all
elements of the form $b^\epsilon a^i$ for $\epsilon \in \{0,1\}$ and $i\in\ZZ$.
A straightforward computation shows that  for every $i\in \ZZ$,
the pairs $(a^i, a^i)$ and $(b a^i, b a^{i+1})$ are elements of $\TFaut(\Theta_n)$,
implying that $H\le \Haut(\Theta_n)$. Moreover, by letting $\psi$  be the automorphism of $H$ that acts
as the identity on $\langle a \rangle$ and maps $ba^i$ to $ba^{i+1}$, we see that
$H$ is $\psi$-stable. 
Finally, $H$ is $\psi$-arc-transitive in the sense of Definition~\ref{def:TF12}.
For future reference, note that when dealing with $\Theta_n$ with $n<\infty$ the order of $\psi$ is $n$.
\end{example}

\begin{corollary}\label{CAction}
Let $H$ be a $\psi$-stable subgroup of $\Haut(\Delta)$. Then
for every $j\in \ZZ_k$ and $g\in H$, the mapping $[g]_{j, \psi}$ defined by
$$(i;x_0, \ldots, x_j, \ldots, x_{k-1})^{[g]_{j, \psi}} = (i;x_0, \ldots,  x_j^{g^{\psi^t}},  \ldots, x_{k-1}),$$
where  $t = (i-j+k-1) \div k$,
is an automorphism of  $\ZZ(\Delta,k)$.  For a fixed value of $j$ this defines an action of $H$ on $\ZZ(\Delta, k)$.
\end{corollary}

\begin{proof}
That $[g]_{j, \psi}$ is an automorphism of $\Gamma=\ZZ(\Delta, k)$ follows immediately by applying Lemma~\ref{LAction} to the sequence $\g=\{g^{\psi^t}\}_{t\in \ZZ}$.
Since $\psi$ is an automorphism of $H$ we get an action of $H$ on $\ZZ(\Delta, k)$.
 \end{proof}

Note that if $g, h \in H$ and $j$ and $j'$ are two distinct elements in $\ZZ_k$, then $[g]_{j, \psi}$ and
$[h]_{j', \psi}$ commute. In particular, the group generated by all the automorphisms $[g]_{j, \psi}$, with $g\in H$ and $j\in \ZZ_k$,
is isomorphic to the group $H^k$. Recall that $s$ denotes the automorphism of $\ZZ(\Delta, k)$ defined by 
$(i;x_0, \ldots, x_{k-2}, x_{k-1})^s = (i+1;x_{k-1}, x_0, \ldots, x_{k-2})$.  Moreover, note that if $0\leq j\leq k-2$ then $([g]_{j, \psi})^s=s^{-1}[g]_{j, \psi}s = [g]_{j+1, \psi}$ but if $j=k-1$ then $([g]_{j, \psi})^s= [g^{\psi^{-1}}]_{0, \psi}$, implying that the group  generated by $s$ and all the elements $[g]_{j, \psi}$,
with $g \in H$ and $j\in \ZZ_k$
(which we will denote by $\tHpsi$, but in the case where $\psi$ is the identity automorphism we will just write $\tH$),
is isomorphic to the semidirect product $H^k\rtimes \langle s\rangle$. If the automorphism $\psi$ has  order $l<\infty$ then the automorphism of $\tHpsi$ induced by conjugation by $s$ has order $kl$.  

\section{$k$-arc transitivity}

The next lemma justifies the introduction of the concepts of $\psi$-stable subgroups and $\psi$-arc-transitive subgroups of $\Haut(\Delta)$.

\begin{lemma}\label{Lkarc}\label{LArc-transitive}
Suppose that $\Delta$ is a digraph such that $\CDHC(\Delta)$ is vertex- and arc-transitive.  Then $\ZZ(\Delta, k)$ is $k$-arc-transitive.  
In the case that $H$ is a $\psi$-arc-transitive, $\psi$-stable subgroup of $\Haut(\Delta)$,
 then the group $\tHpsi$ is $k$-arc-transitive on $\ZZ(\Delta,k)$.
\end{lemma}

\begin{proof}
Since $\aut(\ZZ(\Delta,k))$ is clearly vertex-transitive, it suffices to show that for any
two $k$-arcs $\gamma_1$ and $\gamma_2$ of $\ZZ(\Delta,k)$ with their initial vertex having the first coordinate $0$,
there exists an automorphism mapping $\gamma_1$ to $\gamma_2$. Note that two such $k$-arcs must be of the form
\begin{equation*}
\gamma_1=(0;x_0, x_1, \ldots,x_{k-1}), (1;y_0, x_1, \ldots, x_{k-1}), 
             \ldots,  (k; y_0,y_1, \ldots, y_{k-1}),
\end{equation*}
and
\begin{equation*}
\gamma_2=(0;z_0, z_1, \ldots,z_{k-1}), (1;w_0, z_1, \ldots, z_{k-1}), 
             \ldots,  (k; w_0,w_1, \ldots, w_{k-1}),
\end{equation*}
where, for each $j\in\ZZ_k$,
the pairs $(x_j,y_j)$ and $(z_j,w_j)$ are arcs in $\Gamma$.

Now find for $j=0, \ldots, k-1$ families $\h^j=\{h_t^j\}_{t\in \ZZ}$ of elements in $\Hautde$ such that $(h_t^j, h_{t+1}^j)$ is a a TF-automorphism for all $t\in \ZZ$ and $(x_j^{h_0^j},y_j^{h_1^j}) = (z_j,w_j)$.  By direct calculation one shows that 
$[\h^0]_0 \cdots [\h^{k-1}]_{k-1}$ maps $\gamma_1$ to $\gamma_2$.

The second claim in the Lemma follows directly from the first claim.
\end{proof}

\noindent
{\bf Remark.}  In the case where $\Delta$ is arc-transitive then this Lemma is essentially the same as the first statement in \cite[Proposition~2.11]{Praeger1989} and the proofs use the same basic idea.


\begin{corollary}\label{CArc-transitive}
Let $\Delta$ be a connected arc-transitive digraph.  Suppose that $H\leq \autde$ acts arc-transitively on $\Delta$.  Then the group $G=H^k\rtimes \ZZ$ acts $k$-arc-transitively on $\ZZ(\Delta,k)$, where $\ZZ$ acts on $H^k$ by cyclically permuting the factors.  
\end{corollary}

Let us call a digraph $\Delta$ {\em stable} if for each $x\in \V\Delta$
the pair $\{(x,0), (x,1)\}$ is a block of imprimitivity of $\Aut(\CDHC(\Delta))$.  (A set $A\subseteq \V\Delta$ is called a {\em block of imprimitivity} for a group $G$ acting on $\V\Delta$ if for every $g\in G$ either $A=A^g$ or $A\cap A^g=\emptyset$.)  An automorphism fixing either of the vertices $(x,0)$ or $(x,1)$ must then also fix the other one.


\begin{example}\label{Ecomplete}
Define $\vec{K}_{d}$ as the complete digraph with $d$ vertices ($(v,w)$ is an arc for all vertices $v$ and $w$ with $v\neq w$).  The arc set of $\CDHC(\vec{K}_{d})$ consists of all ordered pairs of vertices $((v,0),(w,1))$ with $v\neq w$.  Hence $\vec{K}_{d}$ is stable.  By the above the graph $\ZZ(\vec{K}_{d}, k)$ is $k$-arc-transitive and has out-valency $d-1$.
\end{example}

Our aim is to show that the digraphs $\ZZ(\Theta_n, k)$ with $3\leq n\leq \infty$  and $\ZZ(\vec{K}_{d+1}, k)$ with $d\geq 2$ are not $(k+1)$-arc-transitive.  This requires a detailed analysis of the structure of these graphs, but a simpler argument can be used to show that these graphs are not highly-arc-transitive (excluding the case $\ZZ(\Theta_\infty, k)$).

\begin{lemma}  \label{LNotHAT}
(a)  The digraphs $\ZZ(\Theta_n, k)$ with $3\leq n<\infty$ are not highly-arc-transitive.

(b)  Let $\Delta$ be a finite connected stable digraph and $k$ a positive integer.
Then $\Gamma = \ZZ(\Delta,k)$ is not highly-arc-transitive.
\end{lemma}

\begin{proof}  We prove part (b).   A similar argument applies to the graphs $\ZZ(\Theta_n, k)$.

Set $G=\autga$.  Recall that the vertex set of $\Gamma$ is equal to the set $\ZZ\times V^k$ and $\Gamma_i=\{i\}\times V^k$ where $V=\V\Delta$.
If $\Delta$ is stable then any automorphism of $\CDHC(\Delta)$ that fixes all the vertices in one part of the bipartition of $\CDHC(\Delta)$ is trivial.  Any automorphism that fixes all the vertices in $\Gamma_0$ must also fix all vertices in $\Gamma_{-1}$ and $\Gamma_1$ and so on.  We conclude that the subgroup $G_{(\Gamma_0)}$ is trivial.  If $v$ is a vertex in $\Gamma_0$ then $G_{(\Gamma_0)}$ has finite index in $G_v$ and hence $G_v$ is finite.  Therefore it is impossible that $G$ acts highly-arc-transitively on $\Gamma$.
\end{proof}

With a bit more work we can do better and show that these graphs cannot be $(k+1)$-arc-transitive.

\begin{theorem}\label{TNot-k+1}
(a)  The digraph $\ZZ(\Theta_n, k)$ for $3\leq n\leq \infty$ is sharply $k$-arc-transitive.  

(b)    Let $\Delta$ be a connected stable directed graph with vertex-set $V$ and $k$ a positive integer.
Assume that whenever $(v,w)$ is an arc in $\Delta$ then $(w,v)$ is also an arc in $\Delta$.
 If $\CDHC(\Delta)$ is arc-transitive  then $\ZZ(\Delta,k)$ is sharply $k$-arc-transitive.
\end{theorem}

\begin{proof}
(a)  Set $\Gamma=\ZZ(\Theta_n, k)$.  The statement that $\Gamma$ is $k$-arc-transitive follows from Lemma~\ref{Lkarc}.

In what follows we assume $n<\infty$ but the argument is the same if $n=\infty$.
Let $\gamma$ be the $k$-arc $(0;0,\ldots,0), (1;0,\ldots,0), \ldots, (k;0,\ldots,0)$. The vertex $(k;0,\ldots,0)$ has two out-neighbours, the vertices $(k+1;0,0,\ldots,0)$ and $(k+1;1, 0,\ldots, 0)$.  We show that the subgroup fixing $\gamma$ must also fix these two vertices.


Suppose $g$ is an automorphism that fixes the $k$-arc $\gamma$ and 
$$(k+1;0, 0,\ldots, 0)^g=(k+1;1, 0,\ldots, 0).$$ 
The alternet containing the vertices $(k;0,0,\ldots, 0)$  and $(k+1;0, 0,\ldots, 0)$ is an alternating $2n$ cycle. The vertex $(k;0,0,\ldots, 0)$ is fixed by $g$ and also, depending on whether $n$ is even or odd,  either the vertex $(k;n/2, 0, \ldots, 0)$ or $(k+1;(n+1)/2,0,\ldots, 0)$ but all the other vertices in the alternet are moved.  Then  $(k;1,0,\ldots, 0)^g=(k;j, 0,\ldots, 0)$ and $j\neq 1$. On the other hand, $g$ fixes all the vertices in the alternet containing the arc $((0,0, \ldots, 0), (1;0, \dots, 0))$. In particular $g$ fixes the vertex $(1;1,0,\ldots, 0)$.  So, $g$ takes the $(k-1)$-arc $(1;1,0, \dots, 0), (2;1,0, \dots, 0), \ldots, (k;1,0, \dots, 0)$ to a $(k-1)$-arc whose initial vertex is $(1;1,0, \dots, 0)$ and terminal vertex is $(k;j,0, \dots, 0)$.  Both these $(k-1)$-arcs are contained in the subdigraph spanned by $\Gamma_1\cup\cdots\cup\Gamma_k$.  If $(i; v_0, \ldots, v_{n-1})$ and $(i+1; w_0, \ldots, w_{n-1})$ are adjacent vertices in this subdigraph then $v_0=w_0$.  Hence it is impossible that there is a $(k-1)$-arc in this subdigraph having $(1;1,0, \dots, 0)$ as an initial vertex and $(k;j,0, \dots, 0)$, with $j\neq 1$,  as a terminal vertex.

Now we have reached a contradiction and it is impossible that such an automorphism $g$ exists.  Whereupon, we see that any automorphism fixing the arc $\gamma$ must also fix the two vertices $(k+1;0,0,\ldots,0)$ and $(k+1;1, 0,\ldots, 0)$ and hence $\ZZ(\Theta_n, k)$ is not $(k+1)$-arc-transitive.

(b)   Set $\Gamma = \ZZ(\Delta,k)$.  The main idea of the proof of part (b) is the same as in the proof of part (a).

Let $v$ and $w$ be adjacent vertices in $\Delta$.  Take the $k$-arc
$$
\gamma=(0;v, v, \ldots,v), (1;w, v, \ldots, v), (2;w, w, \ldots, v), \ldots,  (k; w,w, \ldots, w).
$$
Suppose $g$ is an automorphism fixing each vertex of this $k$-arc and that $g$ does not fix the vertex $(k+1; v,w, \ldots, w)$ that is an out-neighbour of $(k; w,w, \ldots, w)$.  The alternet containing the vertices $(k;w, w, \ldots, w)$ and $(k+1;v, w, \ldots, w)$ is invariant under $g$.  Since $g$ does not fix the vertex $(k+1;v, w, \ldots, w)$ then the stability condition implies that $g$ does not fix $(k;v, w, \ldots, w)$ and $(k;v, w, \ldots, w)^g=(k; u, w, \ldots, w)$ for some $u\neq v$.  
Using the assumption that the digraph $\Delta$ is stable we see that since $g$ fixes  $(0;v, v, \ldots,v)$ the stability condition implies that $g$ also fixes $(1;v, v, \ldots, v)$.   Now we see that $g$ maps the $(k-1)$-arc
$$
(1;v, v, \ldots, v), (2;v, w, \ldots, v), \ldots,  (k; v,w, \ldots, w)
$$
to a $(k-1)$-arc having $(1;v, v, \ldots, v)$ as an initial vertex and $(k;u, w, \ldots, w)$ as a terminal vertex.  Both these $(k-1)$-arcs are contained in the subdigraph spanned by $\Gamma_1\cup\cdots\cup\Gamma_k$.  As above, it is impossible that the latter $(k-1)$-arc exists.  Hence $g$ cannot exist and $\Gamma$ is not $(k+1)$-arc-transitive.
\end{proof}

\begin{corollary}
For every $d\ge 2$ and $k\ge 1$ the graph $\ZZ(\vec{K}_{d+1}, k)$ is a two-ended digraph
 with in- and out-valency equal to $d$ which is sharply $k$-arc-transitive.
\end{corollary}


\section{Variations on a theme}

\subsection{Finite examples of sharply $k$-arc-transitive digraphs}\label{SFinite}  

\medskip

\begin{theorem}\label{TNot-k+1-finite}
(a)   Let $\Delta$ be a connected stable directed graph with vertex-set $V$ and $k$ a positive integer.
Assume that whenever $(v,w)$ is an arc in $\Delta$ then $(w,v)$ is also an arc in $\Delta$ and that  $\CDHC(\Delta)$ is arc-transitive.   Set $\Gamma=\ZZ(\Delta,k)$.
Let $s$ be the automorphism of $\Gamma$ described at the beginning of Section~\ref{SAutomorphism} and set $N=\langle s^q\rangle$ where $q=lk$ for some integer $l>1$.  Then $\Gamma/N$ is a sharply $k$-arc-transitive finite digraph.
 
(b)  Adopt the notation used in Example~\ref{Ecycles} and in Section~\ref{SAutomorphism}.   Define $\Gamma=\ZZ(\Theta_n, k)$, where $3\leq n<\infty$ and $k\geq 1$.  Set  $N=\langle s^q\rangle$ where $q$ is some number that is divisible by $kn$.
Then the digraph $\Gamma/N$ where $N=\langle s^q\rangle$ is a sharply $k$-arc-transitive finite digraph.  
\end{theorem}

\begin{proof}  (a)  Since $s^k$ is central in $\tH$, the group $N$ is normal in $\tH$.  The group $\tH$ acts $k$-arc-transitively on $\Gamma$ and then $\tH$  also acts $k$-arc-transitively on the quotient digraph $\Gamma/N$.
Define $\Gamma_{[0,k+1]}$ as the subdigraph in $\ZZ(\Delta, k)$ that is spanned by $\Gamma_0\cup\cdots\cup\Gamma_k\cup\Gamma_{k+1}$ (see Section~\ref{SConstruction}).  As $q>k$, the image of $\Gamma_{[0,k+1]}$ in the quotient digraph $\Gamma/N$ is isomorphic to   
$\Gamma_{[0,k+1]}$ and the argument used in Theorem~\ref{TNot-k+1} proves that $\Gamma$ is not $(k+1)$-arc-transitive applies.

(b)  Since the order of $\psi$ is $n$ we see that $s^q$ is central in $\tHpsi$.  Hence $N$ is a normal subgroup of $\tHpsi$.  The result then follows in the same way as in part (a).
\end{proof}

\noindent
{\bf Remark.}   In the case where $\Delta$ is a connected arc-transitive graph the graph $\Gamma/N$ described in part (a) of the above Theorem is the same as the graph $C_q(|\V\Delta|,k, \Delta)$ constructed by Praeger in \cite[Definition 2.10]{Praeger1989}.  She proves that these graphs are $k$-arc-transitive, \cite[Proposition 2.1]{Praeger1989}, but does not explore the question whether or not these graphs are $(k+1)$-arc-transitive.  

In her paper Praeger does construct a family of finite digraphs that are sharply $k$-arc-transitive.  She defines a digraph $C_r(v,k)$ that has vertex set $\ZZ_r\times \ZZ_v^k$ and arcs of the type
$$((i; x_1, \ldots,x_k), (i+1;y, x_1, \ldots, x_{k-1}))$$
for $i\in \ZZ_r$ and $y, x_1, \ldots, x_k\in \ZZ_v$, and shows that for $r>k\geq 1$ and $v\geq 2$ the graph $C_r(v, r-k)$ is sharply $k$-arc-transitive \cite[Theorem~2.8(c)]{Praeger1989}.    One can generalize her construction 
and define a digraph $P(v,k)$ with vertex set $\ZZ\times\ZZ_v^k$ and set of arcs defined in the same way as above.  The graph $P(v,k)$ is precisely what you get if you apply the construction described in the remark in Section~\ref{SConstruction} with $\Delta$ a complete graph on $v$ vertices with a loop attached at each vertex.  As observed in \cite[Remark 3.4]{CPW1993}, the digraph $P(v,k)$ is highly-arc-transitive.  Thus, the reason why her digraphs $C_r(v, r-k)$ are not $(k+1)$-arc-transitive is entirely different from the reason why the digraphs constructed in Theorem~\ref{TNot-k+1-finite} are not $(k+1)$-arc-transitive.  

\subsection{Polynomial growth}

Let $\Gamma$ be a connected digraph.  For a vertex $v$ let $b_n(v)$ denote the number of vertices that can be reached from $v$ with a path of length $\leq n$.  The digraph $\Gamma$ is said to have {\em polynomial growth} if there is a polynomial $P$ such that $b_n(v)\leq P(n)$ for all $n\geq 0$.  (One can show that this property does not depend on the choice of the vertex $v$.)   The lowest possible degree of the polynomial $P$ bounding $b_n(x)$ is the {\em degree} of the growth.

 Seifter in \cite{Seifter1991} shows that if $\Gamma$ is an undirected graph with polynomial growth and $\autga$ acts transitively on the set of paths of length $s$ then $s\leq 7$.  As pointed out by Seifter in \cite[p.\ 1532]{Seifter2008} it follows from results in \cite{Moller2002} that a digraph with polynomial growth of degree higher than linear, in which case the digraph has only one end, cannot be highly-arc-transitive.  It is  a natural question to ask if a one-ended digraph with polynomial growth (i.e.~not linear growth) can be $k$-arc-transitive for arbitrarily high values of $k$.

 \begin{corollary}\label{CPolynomial}
The digraph $\ZZ(\Theta_\infty, k)$ has polynomial growth of degree $k+1$ and only one end and is sharply $k$-arc-transitive.
\end{corollary}

\begin{proof}
The digraph $\Theta_\infty$ is a Cayley digraph for $\ZZ$.  By Proposition~\ref{PCayley} the digraph $\ZZ(\Theta_\infty, k)$ is a Cayley digraph for the group $G=\ZZ^k\rtimes \ZZ$ where $\ZZ$ acts on $\ZZ^k$ by cyclically permuting the factors.  If $s$ is a generator for the factor $\ZZ$ in this semi-direct product then $s^k$ is central in $G$ and the subgroup $\ZZ^k\times \langle s^k\rangle\cong \ZZ^{k+1}$ has finite index in $G$.  Using a result of Wolf \cite[Theorem 3.11]{Wolf1968} we conclude that $\Gamma$ has the same growth degree as a locally finite Cayley digraph for $\ZZ^{k+1}$ and by \cite[Proposition~3.6]{Wolf1968} the growth degree of such Cayley digraphs is $k+1$.
The digraph $\ZZ(\Theta_\infty, k)$ has therefore only one end and is locally finite.    It follows from Lemma~\ref{Lkarc} that this digraph is $k$-arc-transitive and from Theorem~\ref{TNot-k+1} that it is not $(k+1)$-arc-transitive.
\end{proof}

Corollary~\ref{CPolynomial} gives examples of sharply $k$-arc-transitive digraphs whose underlying (undirected) graphs have polynomial growth of growth rate $k+1$. So it is clear that the growth rate of our $k$-arc-transitive digraphs grows with $k$. Hence it is natural to ask if one can find digraphs whose underlying graphs have a fixed growth rate $d>1$ but are $k$-arc-transitive for arbitrarily large $k$? Since it seems that an answer to this question might be really difficult to find, we pose it as a problem:
\smallskip

\noindent {\bf Problem.} Does there exist a function $f(d)$ such that if $\Gamma$ is a $k$-arc-transitive digraph  and the underlying graph has growth rate $d$, then  $k<f(d)$?

\subsection{Still more examples with the direct fibre product}

The direct fibre product defined by Neumann in \cite{Neumann2013} can be used in conjunction with the construction of $\ZZ(\Delta, k)$ above.

Let $\Gamma_1$ and $\Gamma_2$ be connected digraphs  having Property Z, witnessed by digraph morphisms $\varphi_1: \Gamma_1\to \vZZ$ and
$\varphi_2: \Gamma_2\to \vZZ$.  The {\em direct fibre product}  $\Gamma=\Gamma_1\,_{\varphi_1}\!\!\times_{\varphi_2}\,\Gamma_2$ is a digraph with vertex and edge sets
\begin{align*}
\V\Gamma&=\{(v_1, v_2)\mid v_1 \in \V\Gamma_1, v_2 \in \V\Gamma_2 \mbox{ and } \varphi_1(v_1)=\varphi_2(v_2)\}\\
\E\Gamma&=\{((v_1, v_2), (w_1, w_2))\mid (v_1, w_1) \in \E\Gamma_1 \mbox{ and } (v_2, w_2) \in \E\Gamma_2\}.
\end{align*}
Let $\pi_1$ and $\pi_2$ denote the projections from $\V\Gamma$ onto the sets $\V\Gamma_1$ and $\V\Gamma_2$, respectively.  Both these projections are digraph homomorphisms.

\begin{lemma}\label{LFibreProduct}  Let $\Gamma_1, \Gamma_2$ and $\Gamma$ be as above.

If both $\Gamma_1$ and $\Gamma_2$ are vertex-transitive then $\Gamma$ is also vertex-transitive.
When both $\Gamma_1$ and $\Gamma_2$ are $k$-arc-transitive then $\Gamma$ is also $k$-arc-transitive.
In particular, if $\Gamma_1$ and $\Gamma_2$ are highly-arc-transitive then $\Gamma$ is highly-arc-transitive.
\end{lemma}

\begin{proof}
Set $G_1=\Aut(\Gamma_1)$ and $G_2=\Aut(\Gamma_2)$.   Let $N_i$ denote the kernel of the action of $G_i$ on the fibers of $\varphi_i$ for $i=1, 2$.  Furthermore, define $\varphi$ as the digraph homomorphism $\varphi:  \Gamma\to \vZZ$ such that $\varphi(v, w)=\varphi_1(v)$ and recall that $\varphi_1(v)=\varphi_2(w)$.

As Neumann observes in \cite{Neumann2013} then $N_1\times N_2$ acts naturally on $\Gamma$ such that if $(g_1, g_2)\in N_1\times N_2$ then $(v,w)^{(g_1, g_2)}=(v^{g_1}, w^{g_2})$. Clearly $N_1\times N_2$ acts transitively on each fiber of $\varphi$.   Say that $g_i\in G_i$ is a {\em translation of magnitude} $l_i$ if $\varphi_i(v^{g_i})=\varphi_i(v)+l_i$ for some vertex $v$ in $\Gamma_i$ (and hence every vertex).   If $g_1\in G_1$ and $g_2\in G_2$ have the same magnitude then the map $(v,w)\mapsto (v^{g_1}, w^{g_2})$ is an automorphism of $\Gamma$.  These two observations imply that $\Gamma$ is vertex-transitive.

Let $\gamma_1$ and $\gamma_2$ be two $k$-arcs in $\Gamma$.  Because we have already shown that $\Gamma$ is vertex-transitive we may assume that $\gamma_1$ and $\gamma_2$ have the same initial vertex.  Set $\gamma_i^j=\pi_j(\gamma_i)$ for $i,j=1,2$.  Then $\gamma_i^j$ is a $k$-arc in $\Gamma_j$ and $\gamma_1^j$ and $\gamma_2^j$ have the same initial vertex.  By assumption there exists an element $g_j\in N_j$ such that $(\gamma_1^j)^{g_j}=\gamma_2^j$.  Now, $(g_1, g_2)\in N_1\times N_2$ and, acting on $\Gamma$, this element takes the $k$-arc $\gamma_1$ to the $k$-arc $\gamma_2$.   The conclusions now follow.
\end{proof}

\begin{theorem}
Suppose $\Gamma_1$ is some connected $k$-arc-transitive digraph that has property Z. 
Let $\Delta$ be a finite connected stable digraph such that $\CDHC(\Delta)$ is arc-transitive and $\ZZ(\Delta, k)$ is connected, or let $\Delta=\Theta_n$ for $3\leq n<\infty$.   Set $\Gamma_2=\ZZ(\Delta, k)$ for some positive integer $k$.  Then 
$\Gamma=\Gamma_1\,_{\varphi_1}\!\!\times_{\varphi_1}\,\Gamma_2$ is 
sharply $k$-arc-transitive.
\end{theorem}

\begin{proof}
That $\Gamma$ is $k$-arc-transitive follows from Lemmas~\ref{Lkarc} and \ref{LFibreProduct}.

Let $L=\ldots, v_{-1}, v_0, v_1, v_2, \ldots$ be an infinite directed line in $\Gamma_1$.  The subdigraph spanned by $\pi_1^{-1}(L)$ is isomorphic to $\Gamma_2=\ZZ(\Delta, k)$.  The argument in Theorem~\ref{TNot-k+1} shows that if an automorphism fixes a $k$-arc lying in this subdigraph then this automorphism must also fix those out-neighbours of the terminal vertex of the $k$-arc that lie in this subdigraph.  Thus $\Gamma$ is not $(k+1)$-arc-transitive.
\end{proof}

\begin{example}
Suppose $\Gamma_2$ is equal to $\ZZ(\Theta_n, k)$, for some $3\leq n<\infty$ or
equal to $\ZZ(\vec{K}_d, k)$, for some $2\leq d<\infty$.

Let $\Gamma_1$ be a directed tree, not isomorphic to $\vZZ$, that is highly-arc-transitive and has Property Z.  Then the direct fibre product of $\Gamma_1$ and $\Gamma_2$ has infinitely many ends and is sharply $k$-arc-transitive.

The Diestel-Leader graphs $\DL(p,q)$, for $p,q\geq 2$ were originally defined in \cite{DiestelLeader2001}.  The Diestel-Leader graphs can be described in various ways; one description is in \cite[Example~1]{Moller2002} and an another one in \cite{BartholdiNeuhauserWoess2008}.    In \cite{BartholdiNeuhauserWoess2008} they are described in terms of the {\em horocyclic product} of trees.  The direct fibre product in \cite{Neumann2013} is an analogue of the horocyclic product and one can describe the Diestel-Leader graphs in terms of the direct fibre product, as explained in \cite[Section 4.5]{Neumann2013}.  Let $T_1$ be the regular directed tree with in-valency 1 and out-valency $p$ and $T_2$ be the regular directed tree with  in-valency $q$ and out-valency 1.  Both of these digraphs have property Z and their direct fibre product is a directed graph $\Gamma_1$ whose underlying undirected graph is the Diestel-Leader graph $\DL(p,q)$.  By Lemma~\ref{LFibreProduct} this digraph $\Gamma_1$ is highly-arc-transitive.  It is shown in \cite[Proposition~5]{DiestelLeader2001} that the graphs $\DL(p,q)$ have just one end but this can also be seen from the construction in \cite[Example~1]{Moller2002}.
The direct fibre product of $\Gamma_1$ and $\Gamma_2$ is a digraph that is sharply $k$-arc-transitive and has one end and exponential growth.
\end{example}

\section*{Acknowledgements}

The first named author wants to thank Alex Wendland for helpful discussions concerning Praeger's construction in \cite{Praeger1989}.


\bibliographystyle{abbrv}
\bibliography{references}

\begin{thebibliography}{10}

\bibitem{BartholdiNeuhauserWoess2008}
L.~Bartholdi, M.~Neuhauser, and W.~Woess.
\newblock Horocyclic products of trees.
\newblock {\em J. Eur. Math. Soc. (JEMS)}, 10(3):771--816, 2008.

\bibitem{CPW1993}
P.~J. Cameron, C.~E. Praeger, and N.~C. Wormald.
\newblock Infinite highly arc transitive digraphs and universal covering
  digraphs.
\newblock {\em Combinatorica}, 13(4):377--396, 1993.

\bibitem{ConderLorimerPraeger1995}
M.~Conder, P.~Lorimer, and C.~Praeger.
\newblock Constructions for arc-transitive digraphs.
\newblock {\em J. Austral. Math. Soc. Ser. A}, 59(1):61--80, 1995.

\bibitem{DiestelJungMoller1993}
R.~Diestel, H.~A. Jung, and R.~G. M\"oller.
\newblock On vertex transitive graphs of infinite degree.
\newblock {\em Arch. Math. (Basel)}, 60(6):591--600, 1993.

\bibitem{DiestelLeader2001}
R.~Diestel and I.~Leader.
\newblock A conjecture concerning a limit of non-{C}ayley graphs.
\newblock {\em J. Algebraic Combin.}, 14(1):17--25, 2001.

\bibitem{HamannPott2012}
M.~Hamann and J.~Pott.
\newblock Transitivity conditions in infinite graphs.
\newblock {\em Combinatorica}, 32(6):649--688, 2012.

\bibitem{LauriMizziScapellato2011}
J.~Lauri, R.~Mizzi, and R.~Scapellato.
\newblock Two-fold automorphisms of graphs.
\newblock {\em Australas. J. Combin.}, 49:165--176, 2011.

\bibitem{LauriMizziScapellato2015}
J.~Lauri, R.~Mizzi, and R.~Scapellato.
\newblock Unstable graphs: a fresh outlook via {TF}-automorphisms.
\newblock {\em Ars Math. Contemp.}, 8(1):115--131, 2015.

\bibitem{Mansilla2007}
S.~P. Mansilla.
\newblock An infinite family of sharply two-arc transitive digraphs.
\newblock {\em Electronic Notes in Discrete Mathematics}, 29:243 -- 247, 2007.
\newblock European Conference on Combinatorics, Graph Theory and Applications.

\bibitem{MansillaSerra2001}
S.~P. Mansilla and O.~Serra.
\newblock Construction of {$k$}-arc transitive digraphs.
\newblock {\em Discrete Math.}, 231(1-3):337--349, 2001.
\newblock 17th British Combinatorial Conference (Canterbury, 1999).

\bibitem{Moller2002}
R.~G. M\"oller.
\newblock Descendants in highly arc transitive digraphs.
\newblock {\em Discrete Math.}, 247(1-3):147--157, 2002.

\bibitem{MollerPotocnikSeifter2018}
R.~G. M\"oller, P.~Poto\v{c}nik, and N.~Seifter.
\newblock Infinite arc-transitive and highly-arc-transitive digraphs.
\newblock {\em submitted}.

\bibitem{Neumann2013}
C.~Neumann.
\newblock Constructing highly arc transitive digraphs using a direct fibre
  product.
\newblock {\em Discrete Math.}, 313(23):2816--2829, 2013.

\bibitem{Praeger1989}
C.~E. Praeger.
\newblock Highly arc transitive digraphs.
\newblock {\em European J. Combin.}, 10(3):281--292, 1989.

\bibitem{Seifter1991}
N.~Seifter.
\newblock Properties of graphs with polynomial growth.
\newblock {\em J. Combin. Theory Ser. B}, 52(2):222--235, 1991.

\bibitem{Seifter2008}
N.~Seifter.
\newblock Transitive digraphs with more than one end.
\newblock {\em Discrete Math.}, 308(9):1531--1537, 2008.

\bibitem{ThomassenWoess1993}
C.~Thomassen and W.~Woess.
\newblock Vertex-transitive graphs and accessibility.
\newblock {\em J. Combin. Theory Ser. B}, 58(2):248--268, 1993.

\bibitem{Weiss1981}
R.~Weiss.
\newblock The nonexistence of {$8$}-transitive graphs.
\newblock {\em Combinatorica}, 1(3):309--311, 1981.

\bibitem{Wolf1968}
J.~A. Wolf.
\newblock Growth of finitely generated solvable groups and curvature of
  {R}iemanniann manifolds.
\newblock {\em J. Differential Geometry}, 2:421--446, 1968.

\bibitem{Zelinka1971}
B.~Zelinka.
\newblock The group of autotopies of a digraph.
\newblock {\em Czechoslovak Math. J.}, 21(96):619--624, 1971.

\end{thebibliography}

\end{document}